\documentclass[a4paper,reqno,11pt]{amsart}

\usepackage{amsmath}
\usepackage{amssymb}
\usepackage{amsthm}
\usepackage{a4wide}
\usepackage{enumerate}
\usepackage{url}

\newcommand{\R}{\mathbf{R}}
\newcommand{\Z}{\mathbf{Z}}
\newcommand{\Q}{\mathbf{Q}}
\newcommand{\C}{\mathbf{C}}
\newcommand{\Qbar}{\overline{\Q}}
\newcommand{\GQ}{\Gal(\Qbar/\Q)}
\DeclareMathOperator{\Gal}{Gal}
\DeclareMathOperator{\GL}{GL}
\DeclareMathOperator{\PGL}{PGL}
\DeclareMathOperator{\SL}{SL}
\DeclareMathOperator{\trace}{trace}

\DeclareMathOperator{\Frob}{Frob}
\DeclareMathOperator{\SO}{SO}
\DeclareMathOperator{\disc}{disc}
\DeclareMathOperator{\cond}{cond}

\theoremstyle{plain}
\newtheorem{theorem}{Theorem}
\newtheorem{lemma}[theorem]{Lemma}

\theoremstyle{remark}

\begin{document}

\title[A computation of modular forms of weight one and small level]{A computation of modular forms of \\ weight one and small level}

\begin{abstract}
We report on a computation of holomorphic cuspidal modular forms of weight one and small level (currently
level at most $1500$) and classification of them according to the projective image of their attached Artin
representations. The data we have gathered, such as Fourier expansions and projective images of Hecke newforms
and dimensions of space of forms, is available in both Magma and \texttt{Sage} readable formats on a webpage
created in support of this ongoing project. We explain some of the novel aspects of these computations and what they
have uncovered.
\end{abstract}

\author{Kevin Buzzard}\email{k.buzzard@imperial.ac.uk}\address{Department of
Mathematics, Imperial College London, 180 Queen's Gate, London SW7 2AZ, England}
\author{Alan Lauder}\email{lauder@maths.ox.ac.uk}\address{Mathematical Institute, Woodstock Road, Oxford OX2 6GG, England}
\maketitle

\section{Introduction}

The theory and practice of computing weight one modular forms has typically lagged behind that of computing higher weight forms. This is mainly because forms of higher weight are cohomological and there is a direct method for computing them using modular symbols. Indeed, the computer algebra package Magma~\cite{magma} has been able to compute forms of weight two or more for over fifteen years now using modular symbols, and the free open source package \texttt{Sage}~\cite{sage} can also compute these forms.
(Custom code existed well before then: examples we know of are due to Cohen--Skoruppa--Zagier, Cremona, Gouv\^ea and Stein.) By contrast, no such direct method is known in weight one and there were no generally applicable algorithms until more recently --- the pioneering work of Buhler~\cite{buhler} and project coordinated by Frey~\cite{freyetal} were both focused on computing one or more specific spaces of forms, rather than on a systematic computation.

The first author adapted the methods of Buhler and Frey et al.\ so that they could be applied systematically, and reported on the details of the algorithm in~\cite{buzzard:wt1}. This code, which computed bases of spaces of modular forms of weight one
and arbitrary Dirichlet character, was written in Magma and incorporated into the distributed version of the Magma package by Steve Donnelly. The authors have used this code (with some additions) to carry out a computation of the Hecke newforms in weight one for increasing level and all characters.
Computations have been completed for all levels up to $1500$ and the data obtained
is available in both Magma and \texttt{Sage} formats on a webpage which accompanies this paper \cite{weight1-page}. 
George Schaeffer informs the authors that he has implemented his more efficient ``Hecke stability method'' \cite{hecke-stability} for weight one in \texttt{Sage}, but only for quadratic character, and has computed such newforms up to level around~800. A Magma implementation of Schaeffer's algorithm for general character would be of great practical use in extending our tables of weight one modular 
forms (at present \texttt{Sage} seems a less suitable platform for carrying out such computations).

What do we mean by {\it computing weight one newforms}? For a given level~$N$ and odd character~$\chi$, we present
each cuspidal new eigenform $f \in S_1(N,\chi)$ as a truncated $q$-expansion $f(q)+{\mathcal O}(q^M)$ with Fourier coefficients in an 
explicit cyclotomic field containing the image of~$\chi$.
The $q$-adic precision $M$ is chosen so that there is a unique weight two form of level $N$ and trivial character whose
$q$-expansion is $E_1(1,\chi^{-1}) \cdot f(q) + {\mathcal O}(q^M)$, where $E_1(1,\chi^{-1})$ denotes
the Eisenstein series of weight one and characters $1$ and $\chi^{-1}$, thus ensuring further Fourier coefficients
can be easily computed if desired using modular symbols in weight two. On our webpage bases of spaces of modular forms
are given in the same manner, and we provide code which allows the user to compute the Fourier expansions
to arbitrary precision, as well as computations of $q$-expansions up to $O(q^{10000})$.

Having computed all newforms up to a given level, two natural questions for us to consider were what further computations can one do with this data, and how to make the data available to other researchers in an easily accessible manner. 

The main computation we did with the data was that for each cuspidal newform we rigorously computed whether the projective image of the associated Galois representation was a dihedral group or one of $A_4$, $S_4$ or $A_5$. This seemed like a natural question to ask and it needed, what was for us, a novel trick to answer.
Note that as a consequence we are able to determine the {\it smallest} level $N$ for which there exists a weight one modular form whose associated projective Galois representation has image~$A_5$. The level is $633$ and the Dirichlet character has order $10$.
The analogous questions for $A_4$ and $S_4$ were answered in~\cite{buzzard:wt1}, levels $124$ and $148$ with characters of order $6$ and $4$, respectively.
This level~633 icosahedral form does not seem to have been computed before -- Buhler's original icosahedral example had level $800$ and the first author in~\cite{buzzard:wt1} found an example with level~675. 
(The original motivation for the second author in carrying out such computations was
for a specific experimental application which required knowledge of the projective image. Such a classification of cuspidal newforms of small level in weight one was crucial in developing and numerical testing the conjectural, and occasionally provable, new constructions of points on elliptic curves and units
defined over dihedral, $A_4$, $S_4$ and $A_5$ number fields in \cite{DLR1,DLR2}. The $A_4$ form of level 124, and the $S_4$ form of level $148$ occur
in \cite[Examples 5.4 and 5.6]{DLR1}.)

One further computation which could be done with the data is to find the number field cut out by the projective Galois representation associated to each
cuspidal newform. The most straightforward way to answer this question for a given newform is to search in a pre-computed table of number fields. Once a number field has been found that one suspects is the right one, one can rigorously prove that it is by invoking the Artin conjecture, which is known in this situation thanks to the work of Khare and Wintenberger \cite{KW1}. For example,
the number field cut out by the projective Galois representation attached to the $A_5$ form in level $633$ is the splitting field of the polynomial $x^5 - 211x^2 - 1266x - 1899$.
We did not attempt to automate this process though. (An alternative analytic approach to finding candidate number fields, working directly from the Fourier expansions, is to
invoke Stark's conjecture \cite{Stark}. This may also be used to find the maximal real subfield of the field cut out by the Galois representation itself.)

We have made available on a webpage all of our data (on Fourier expansion of cuspidal newforms, dimensions and bases of spaces of cuspidal forms) in easily readable Magma and \texttt{Sage} format, along with accompanying code 
which allows the user to perform some further computations. This seemed to the authors the best way of making the
data accessible and useable for other researchers.

The remainder of this note is devoted to explaining how the projective image of each cuspidal newform was 
rigorously determined.

\section{The Galois representation associated to a weight~1 form.}

In this section we give a brief overview of the relationship between weight~1 modular forms and Galois representations. All of the material here is well-known, and our exposition is mostly to set up notation.

Let $N\geq1$ and $\chi$ be a Dirichlet character modulo~$N$. The space of weight~1 forms of level~$\Gamma_1(N)$ and character~$\chi$ is finite-dimensional over the complex numbers and will be denoted $S_1(N,\chi)$. Let $d(N,\chi)$ denote its dimension. If $K$ is a subfield of $\C$ containing $\Q(\chi)$, the field generated by the image of~$\chi$, then we write $S_1(N,\chi;K)$ for the $K$-vector space of forms in $S_1(N,\chi)$ whose $q$-expansions are in $K[[q]]$. The $K$-dimension of this space is $d(N,\chi)$. 

If $f=\sum_{n\geq1}a_nq^n\in S_1(N,\chi)$ then we say $f$ is a \emph{normalised eigenform} if $f$ is an eigenform for all the Hecke operators (including those at the bad primes) and $a_1 = 1$. In this case the eigenvalue of the Hecke operator $T_n$ is the complex number $a_n$. 

There is a theory of oldforms and newforms, which works in weight~1 just as in other weights. Hecke operators $T_p$ at primes $p$ dividing~$N$ may not be diagonalisable on $S_1(N,\chi)$, but they are diagonalisable on the new subspace. If $f=\sum a_nq^n\in S_1(N,\chi)$ is a normalised eigenform which is furthermore a newform of level~$N$, then a theorem of Deligne and Serre (Theorem~4.1 of~\cite{deligne-serre}) tells us that there is a continuous odd irreducible Galois representation
$$\rho_f:\GQ\to\GL_2(\C)$$
associated to~$f$ (by ``odd'' we mean that if $c$ is complex conjugation then the determinant of $\rho_f(c)$ is $-1$). In contrast to the higher weight case, the target here is a complex group rather than a $p$-adic one, and furthermore the image of $\rho_f$ is finite, so $\rho_f$ can be thought of as a faithful representation $\Gal(M_f/\Q)\to\GL_2(\C)$, where $M_f$ is a finite Galois extension of~$\Q$. The representation~$\rho_f$ has conductor~$N$, and in particular the extension $M_f/\Q$ is unramified outside~$N$. The representation $\rho_f$ is characterised by the following property: if $p\nmid N$ is prime, then the characteristic polynomial of $\rho_f(\Frob_p)$ is $X^2-a_pX+\chi(p)$. Here there is a choice: all Frobenii can either be arithmetic or geometric, and it does not matter which convention we use in this paper as long as we are consistent, so the reader may feel free to choose their favourite. The Cebotarev density theorem tells us that every conjugacy class in $\Gal(M_f/\Q)$ is of the form $\Frob_p$ for infinitely many $p\nmid N$, so in particular giving the characteristic polynomials of the $\Frob_p$ determines (and indeed highly overdetermines) the representation $\rho_f$. 

We also know the behaviour of $\rho_f$ at primes dividing~$N$. At these primes the relationship between the local behaviour of~$f$ and~$\rho_f$ is given by the local Langlands correspondence; this is because the construction of $\rho_f$ is known to satisfy local-global compatibility. Indeed the $L$-function of $\rho_f$ is equal to the Mellin transform of~$f$ by Theorem~4.6 of~\cite{deligne-serre}, and applying this result to~$f$ and its twists tells us enough about local~$L$ and~$\varepsilon$ factors to deduce the full local-global correspondence.\footnote{We remark that this local-global compatibility is still an open problem in the analogous situation of Hilbert modular forms for which some but not all weights are~1; see for example~\cite{newton} for the current state of the art.}

As a consequence of this argument we see that if~$f$ is a normalised newform then there is a finite cyclotomic extension~$\Q(\zeta)$ of $\Q$ such that $a_p\in\Q(\zeta)$ for all primes $p\nmid N$ -- indeed, for those $p$ we have that $a_p$ is a sum of two roots of unity. At primes~$p$ dividing the level the situation is even simpler -- either $a_p=0$ or $a_p$ is a root of unity (as $a_p$ is the trace of Frobenius on the subspace of the underlying 2-dimensional complex vector space fixed by an inertia subgroup at~$p$, by the local-global correspondence). The usual formulae relating $a_n$ for a general~$n$ to $a_p$ for $p$ prime apply in weight~1, and we deduce that the field generated by the coefficients of~$f$ lies in a cyclotomic field which by a standard argument (consider $a_p$ and $a_{p^2}$) contains $\Q(\chi)$, the number field generated by the image of $\chi$.

We can projectivise $\rho_f$ and obtain a projective Galois representation
$$\overline{\rho}_f:\GQ\to\PGL_2(\C).$$
The image of $\overline{\rho}_f$ is a finite subgroup of $\PGL_2(\C)$ and is hence either cyclic, dihedral (including the degenerate case of the non-cyclic group of order~4), or isomorphic to $A_4$, $S_4$ or $A_5$. (This argument goes back to Weber; the pre-image of a finite subgroup of $\PGL_2(\C)$ in $\SL_2(\C)$ stabilises a hermitian form so lives in $SU(2)$, which maps in a 2-to-1 manner onto $\SO_3(\R)$, and the finite subgroups of this can be classified via the Platonic solids.) Because $\rho_f$ is irreducible, the image of $\overline{\rho}_f$ cannot be cyclic (otherwise the image of $\rho_f$ would be a central extension of a cyclic group by a cyclic group and hence abelian), but the other cases do occur. We refer to $f$ in these cases as a dihedral form, an $A_4$ form, an $S_4$ form or an $A_5$ form respectively. In the next section we give an algorithm for determining, for a given newform~$f$, whether it is dihedral, $A_4$, $S_4$ or $A_5$. In the remainder of this section we prove some relatively straightforward lemmas, which are used in the proof of correctness of the algorithm.

\begin{lemma}\label{orders} (a) If $g\in\PGL_2(\C)$ has finite order~$n$, and $\tilde{g}\in\GL_2(\C)$ is any lift of $g$, then the complex number $c(\tilde{g})=\trace(\tilde{g})^2/\det(\tilde{g})$ is independent of the choice of $\tilde{g}$, and writing $c(g)$ for $c(\tilde{g})$ we have $c(g)=2+\zeta+\zeta^{-1}$ where $\zeta\in\C$ is a primitive $n$th root of unity.

(b) If $g$ has order 1,2,3,4 then $c(g)=4,0,1,2$ respectively. If $g$ has order~5 then $c(g)=\frac{3\pm\sqrt{5}}{2}$.
\end{lemma}
\begin{proof} 
(a) Any two lifts of $g$ to $\GL_2(\C)$ differ by a non-zero scalar, from which it is easy to check that $c(\tilde{g})$ depends only on~$g$. If $\tilde{g}$ is any lift of~$g$ to $\GL_2(\C)$ then $\tilde{g}$ must be diagonalisable (or else~$g$ would have infinite order) and if the eigenvalues are $\lambda$ and $\mu$ then the order of~$g$ in $\PGL_2(\C)$ equals the multiplicative order of $\lambda/\mu$ in $\C^\times$. Because~$g$ has order~$n$, we must have $\lambda=\mu\zeta$ for $\zeta$ a primitive $n$th root of unity, and now everything follows from a direct calculation.

(b) This follows easily from (a).
\end{proof}

Here is a basic analysis of dihedral forms (by which we recall that we means forms~$f$ such that $\overline{\rho}_f$ is a dihedral group, including the degenerate case of the non-cyclic group of order~4). 

\begin{lemma}\label{dihedral} Let $f$ be a dihedral form, and suppose $\rho_f$ cuts out the extension $\Gal(M/\Q)$ for some number field~$M$. Then there exists a quadratic extension $K/\Q$ contained in~$M$ and a finite order character $\psi:\Gal(\overline{K}/K)\to\C^\times$ such that $\rho_f$ is the 2-dimensional representation induced from $\psi$. The trace of~$\rho_f$ is zero on every element of $\Gal(M/\Q)$ which is not in $\Gal(M/K)$.
\end{lemma}

\begin{proof} The image of~$\overline{\rho}_f$ is dihedral and corresponds to a quotient $\Gal(L/\Q)$ of $\Gal(M/\Q)$. This dihedral group contains an index~2 cyclic subgroup, which by Galois theory corresponds to a quadratic extension~$K$ of~$\Q$ contained in~$L$ and hence in~$M$. The image of $\Gal(M/K)$ under $\overline{\rho}_f$ is cyclic, and hence the image of $\Gal(M/K)$ under $\rho_f$ is a central extension of a cyclic group by a cyclic group and hence abelian. By Schur's Lemma this means that the restriction of $\rho_f$ to $\Gal(M/K)$ is the sum of two characters $\psi$ and $\overline{\psi}$ (the conjugate of $\psi$ by the automorphism of $\Gal(M/K)$ induced by the non-trivial field automorphism of~$K$). Now because $\rho_f$ is irreducible it must (by Frobenius reciprocity) be the isomorphic to the representation of $\Gal(M/\Q)$ induced by $\psi$. A standard calculation then shows that the trace of $\rho_f$ vanishes outside $\Gal(M/K)$.
\end{proof}

\begin{lemma}\label{A4notS4} Suppose $f$ is an $S_4$ form and $\rho_f$ cuts out the extension $\Gal(M/\Q)$. Then there exists a quadratic extension $K/\Q$ contained within~$M$ such that for every element~$\sigma$ of $\Gal(M/\Q)$ which is not in $\Gal(M/K)$, the order of $\overline{\rho}_f(\sigma)$ is either~2 or~4.
\end{lemma}
\begin{proof} This follows immediately by Galois theory from the fact that every element of $S_4$ which is not in the index~2 subgroup~$A_4$ has order~2 or~4.
\end{proof}

\begin{lemma}\label{A4notA5} Suppose $f$ is an $A_4$ form and $f$ has character of order coprime to~5. Then the coefficient field of~$f$ does not contain $\sqrt{5}$. 
\end{lemma}
\begin{proof} Let $G$  be the image of $\overline{\rho}_f$, so~$G$ is a subgroup of $\PGL_2(\C)$ isomorphic to~$A_4$. The pre-image of~$G$ in the degree~2 cover $\SL_2(\C)$ of $\PGL_2(\C)$ is then a group $\tilde{G}$ of order~24, and if~$Z$ denotes the scalar matrices in $\GL_2(\C)$ then the image of $\rho_f$ must be contained within the group $Z\tilde{G}$. Furthermore, because the determinant of $\rho_f$ equals the character of~$f$ and in particular has order $d$ prime to~5, the image of $\rho_f$ must be contained within $\mu_{2d}\tilde{G}$ where $\mu_{2d}$ denotes the $2d$th roots of unity within $Z$.
In particular the image of $\rho_f$ must have order prime to~5. If $f=\sum a_nq^n$ then this means that each $a_p$ is a sum of at most~2 roots of unity of order prime to~5 (this is true even at the bad primes: by local-global compatibility the $a_p$ in this case equals the trace of Frobenius on the inertial invariants) and the field generated by the $a_p$ and the values of $\chi$ is hence unramified at~5.
\end{proof}

\begin{lemma}\label{sturm} Let $f$ and~$g$ be weight~1 modular forms with the same level~$N$ and character~$\chi$. Then $f=g$ if and only if the $q$-expansions of~$f$ and~$g$ agree up to and including the term $q^M$ with $M=\frac{1}{12}[\SL_2(\Z):\Gamma_0(N)]$ (the so-called ``Sturm bound'' for level~$N$).
\end{lemma}
\begin{proof} This follows from Corollary~1.7 of~\cite{buzzard-stein} (applied with $I$ equal to the empty set), but it is a standard argument and we sketch it here. Conceptually what is going on is that $f$ and~$g$ are sections of a line bundle $\omega$ on the modular curve $X_1(N)$, so if their $q$-expansions agree up to a point greater than the degree of $\omega$ then their difference is a holomorphic section of $\omega$ with a zero of such a large degree that it forces the difference to be zero. Actually this is computationally ineffective because the degree of $\omega$ scales as $N^2$. The way to obtain the lemma with the constants as stated is to apply that argument to $(f-g)^d$ where $d$ is the order of $\chi$; this is a weight~$d$ form of level $\Gamma_0(N)$ and an explicit computation of the degree of $\omega^d$ on $X_0(N)$ gives the result.
\end{proof}
\section{Computing the projective image.}

Let $f=\sum_{n\geq1}a_nq^n$ be a normalised cuspidal newform of weight~1, level~$N$ and character~$\chi$, and let $\rho_f:\Gal(M/\Q)\to\GL_2(\C)$ (assumed injective) be the associated Galois representation. The Cebotarev density theorem ensures that each conjugacy class in $\Gal(M/\Q)$ equals the Frobenius conjugacy class $\Frob_p$ for infinitely many primes~$p$, so we know that for any $\sigma\in\Gal(M/\Q)$ the characteristic polynomial of $\rho_f(\sigma)$ will equal $X^2-a_pX+\chi(p)$ for infinitely many primes~$p$. We are interested in this section in computing the isomorphism class of $\overline{\rho}_f(\Gal(M/\Q))$ and in particular whether it is dihedral or isomorphic to $A_4$, $S_4$ or~$A_5$. In practice it is relatively easy to \emph{guess} the answer, simply by looking at the $q$-expansion of~$f$. If about half of the coefficients $a_p$ ($p$ prime) equal zero then~$f$ will probably be dihedral. If not, then $\overline{\rho}_f(\sigma)$ has order at most~5 for all $\sigma\in\Gal(M/\Q)$ and if $\sigma=\Frob_p$ then the invariant $c(\sigma)$ of Lemma~\ref{orders} equals $a_p^2/\chi(p)$. The numbers $a_p^2/\chi(p)$ can easily be computed for the first few hundred primes~$p\nmid N$, and $a_p^2/\chi(p)$ tells us the order of $\overline{\rho}_f(\Frob_p)$ by Lemma~\ref{orders}(b). If we see only elements of order at most~3 then $f$ is probably an $A_4$ form, if we see an element of order~4 then $f$ is an $S_4$ form (as neither~$A_4$ nor~$A_5$ have elements of order~4) and if we see an element of order~5 then $f$ is an $A_5$ form (as neither $A_4$ nor $S_4$ have elements of order~5). Note that 25\% of the elements of~$S_4$ have order~4 and 40\% of the elements of~$A_5$ have order~5, so by the Cebotarev density theorem it should be very easy in practice to find such elements and we would only expect to have to try a few primes.

To make this algorithm rigorous we see that we have to solve two problems. Firstly, we need to be able to rigorously decide whether or not a form~$f$ is dihedral. Secondly, we need to have a way of proving that a form which we know to be non-dihedral and which we suspect of being an $A_4$ form, is actually an $A_4$ form (we can deal with $S_4$ and $A_5$ rigorously by finding elements of order~4 and~5 respectively, once we have proven that the form is not dihedral). We deal with these two issues separately below.

\subsection{An algorithm to decide whether or not a form is dihedral.}

Here is a method which rigorously decides whether or not a form $f=\sum_n a_nq^n$ of level~$N$ and character~$\chi$ is dihedral. By Lemma~\ref{dihedral}, if~$f$ is dihedral then there exists some quadratic extension $K/\Q$ contained in~$M$ (and hence unramified outside~$N$) such that the trace of $\rho_f(\sigma)$ is zero whenever $\sigma\in\Gal(M/\Q)$ does not fix~$K$. Hence there is a quadratic extension $K/\Q$ unramified outside~$N$ such that for every prime $p\nmid N$ which is inert in~$K$, $a_p=0$. Thus if we suspect that~$f$ is \emph{not} dihedral, an algorithm for proving that it is not dihedral is the following. List all quadratic extensions $K/\Q$ unramified outside~$N$ (there are of course only finitely many) and for each one, search for a prime $p\nmid N$ which is inert in~$K$ and for which $a_p\not=0$. If for each $K$ we can find such a prime then we have proved that~$f$ is not dihedral. One can check using the Cebotarev density theorem that such primes will be extremely easy to find if~$f$ is not dihedral, as $A_4$, $S_4$ and $A_5$ only have at most $9/24$ of their elements of order~2, so in practice one expects only to have to check a few primes before one finds an example.

Next we need to give an algorithm for proving that a form that we suspect \emph{is} dihedral, is actually dihedral. By Lemma~\ref{dihedral}, if~$f$ is dihedral then $\rho_f$ is induced from a character $\psi$ of $\Gal(\overline{K}/K)$ for some $K/\Q$ unramified outside~$N$. Conversely, if $\psi$ is a character of $\Gal(\overline{\Q}/K)$ for some quadratic extension~$K$ of~$\Q$ then its induction to $\Gal(\overline{\Q}/\Q)$ is a 2-dimensional representation. If furthermore this representation is irreducible and odd, then it comes from a modular form; this is a special case of Artin's conjecture (a theorem of Khare and Wintenberger) but this case has been known for far longer and a convenient reference is Theorems~4.8.2 and~4.8.3 of~\cite{miyake}, which furthermore shows how to explicitly recover the level, the character, and the $q$-expansion of the form from~$\psi$. Note in particular from these theorems that the level of the newform giving rise to the induced representation equals the product of the discriminant of~$K$ and norm of the conductor of~$\psi$.

So here is our algorithm, to prove that a suspected-dihedral form really is dihedral: we loop over quadratic fields~$K$ unramified outside~$N$ and for each such field we loop over all characters~$\psi$ of $\Gal(\overline{K}/K)$ with conductor of norm equal to $N/\disc(K)$. For each such character $\psi$ having the additional properties that $\psi$ is not equal to its Galois conjugate (this is to ensure that the induced representation is irreducible) and that the induced representation is odd (this is automatic if $K$ is imaginary, and in the real case it is a condition at infinity) we compute the $q$-expansion of the corresponding weight~1 form (using Theorems~4.8.2 and~4.8.3 of~\cite{miyake}). All of these $q$-expansions are guaranteed to be weight~1 modular forms. Eventually we will find a $q$-expansion of a form with the same character as that of~$f$, such that the $q$-expansion looks equal to the $q$-expansion of~$f$ to the degree of accuracy that we have computed it. We now need to prove that these forms are equal, and we do this by computing the $q$-expansions of the forms up to the Sturm bound and applying Lemma~\ref{sturm}.

There are other approaches to showing that a suspected-dihedral form is dihedral. For example, if $f$ is a modular form and we have found a quadratic extension $K/\Q$ such that $a_p$ seems to be 0 for all of the $p\nmid N$ such that $p$ is inert in $K$ (in the sense that we check this for hundreds of primes and it is true in every case) then we could let $\xi$ denote the quadratic Dirichlet character corresponding to~$K$ by class field theory and then verify that the newform corresponding to $f\otimes\xi$, that is the newform corresponding to $\sum a_n\xi(n)q^n$, equalled~$f$. The problem with this method is that we would verify the equality using Lemma~\ref{sturm}, and $f\otimes\xi$ could in theory have level $Nd^2$ with $d$ the conductor of $\xi$, and if~$d$ is large then it makes the Sturm bound much larger and hence demands the computation of far more $q$-expansion coefficients. This twisting method might be appropriate if the explicit class field theory calculations had turned out to be difficult, but magma had no problem at all with these class field calculations because in practice it was only having to work with quadratic fields with discriminant of order of magnitude at most~1500. Of course implementing the class field theory approach was far harder, however this implementation was already part of the algorithm used to compute the space of weight~1 forms in the first place and so in practice it was there already.

\subsection{Verifying that a probably-$A_4$ form is an $A_4$ form.}

The only remaining problem is the following. Suppose we have a form~$f$ which we have proved is non-dihedral and we suspect is an $A_4$ form. We need to rule out~$f$ being of types~$S_4$~or~$A_5$. We have computed $a_p^2/\chi(p)$ for many unramified primes~$p$ and hence the order of $\overline{\rho}_f(\Frob_p)$ for many primes~$p$. We have found no Frobenii with order~4 or~5. How can we \emph{prove} that~$f$ is an $A_4$ form, rather than there being a Frobenius element of order~4 or~5 just around the corner? Of course one could attempt to use an explicit form of the Cebotarev density theorem, but this would involve computing things about the number field cut out by the projective representation and the problem is that we do not provably know what this number field is at this point; indeed we do not even know its Galois group over~$\Q$ at this point in the argument, or even its degree. We resolved this problem in a completely different way which turned out to also be far more efficient.

Because~$f$ is known to be non-dihedral, to prove that it is an~$A_4$ form all we have to do is to prove that it is not an~$S_4$ form or an~$A_5$ form. Here are the two tricks.

Proving that a suspected-$A_4$ form~$f$ not actually an~$A_5$ form is usually easy in practice. By the Cebotarev density theorem, if~$f$ is an~$A_5$ form then there really will be primes~$p\nmid N$ such that $a_p^2/\chi(p)=\frac{3\pm\sqrt{5}}{2}$ and in particular the coefficient field of~$f$ must contain~$\sqrt{5}$. In our range of computation it turned out that this observation already did the job -- for all $A_4$ forms of level~$N\leq1500$ the coefficient field did not contain $\Q(\sqrt{5})$. However there are $A_4$ forms whose field of coefficients contains~$\sqrt{5}$ -- for example take an $A_4$ form and twist it by a Dirichlet character to ensure that it has character of order a multiple of~5; then the coefficient field of the twisted form will contain $\Q(\zeta_5)$ and this contains $\sqrt{5}$. We can fix this problem however, were it ever to occur, using the following trick: if $f$ is any eigenform then some twist $f'=f\otimes\xi$ of~$f$ will have character $\chi'$ of order coprime to~5 (if $f$ has character $\chi$ of order $5^de$ with $d\geq1$ and $5\nmid e$ then twist $f$ by $\chi^{(5^d-1)/2}$ for example). The coefficient field of $f'$ is certainly contained in the compositum~$L$ of $\Q(\xi)$ and the coefficient field of~$f$, but it might be smaller than this. We can compute the coefficient field explicitly however. For it is contained within the Galois closure~$L'$ of the field generated by the image of $\chi'$ and the first $M$ terms in the $q$-expansion of~$f'$, where $M$ is the Sturm bound for level $N\cond(\xi)^2$. The level of $f'$ divides $N\cond(\xi)^2$, and if the true coefficient field contains an element not in~$L'$ then we can find an automorphism of $\overline{\Q}$ which fixes~$L'$ and $\chi'$ but sends $f'$ to $f''\not=f'$; this however contradicts Lemma~\ref{sturm}. 

The point of this twisting is that if $f$ really is an $A_4$ form, then applying Lemma~\ref{A4notA5} to $f'$ we see that $L$ will be unramified at~5, and conversely if~$f$ is an $A_5$ form then $L$ must be ramified at~5, so this test if run on an $A_4$ form is guaranteed to prove that it is not an $A_5$ form.

It is hard to comment on the effectiveness of this algorithm because we never had to use it; the coefficient field of~$f$ never had $\sqrt{5}$ in the $A_4$ cases in the range that we computed. However given that we computed the first 10,000 coefficients of the $q$-expansions of all the forms of level at most~1500 one should be optimistic that our suggested algorithm will succeed in practice if it is ever needed.

Our final task is to give an algorithm which proves that a suspected-$A_4$ form is provably not an $S_4$ form. This caused us some trouble for a while, however the argument we ultimately found is very short. Here is how it works. Let us say that actually~$f$ is an $S_4$ form. Then by Lemma~\ref{A4notS4} and Lemma~\ref{orders} there is a quadratic extension $K/\Q$ unramified outside~$N$ such that for each $p\nmid N$ which is inert in~$K$, we have $a_p^2/\chi(p)\in\{0,2\}$. So we simply loop through all such quadratic fields and for each field we find an unramified inert prime~$p$ which does not have this property. If~$f$ really is an $A_4$ form then again the Cebotarev density theorem guarantees that for each~$K$ this will happen very quickly, and we should only have to check a few primes. Our description of the algorithm is complete.

\bibliographystyle{amsalpha}

\providecommand{\bysame}{\leavevmode\hbox to3em{\hrulefill}\thinspace}
\providecommand{\MR}{\relax\ifhmode\unskip\space\fi MR }
\providecommand{\MRhref}[2]{%
  \href{http://www.ams.org/mathscinet-getitem?mr=#1}{#2}
}
\providecommand{\href}[2]{#2}

\end{document}